\documentclass{amsart}

\usepackage{latexsym,amsfonts} 
\usepackage{amsmath,amssymb,graphics,setspace}
\usepackage{amsthm}
\usepackage{MnSymbol}
\usepackage{mathrsfs} 
\usepackage{fontenc}%

\usepackage{marvosym}
\usepackage{paralist}
\usepackage{color}

\usepackage{pstricks,pst-text,pst-grad,pst-node,pst-3dplot,pstricks-add,pst-poly} 

\usepackage{graphicx}
\usepackage{fontenc}%

\usepackage{hyperref}  
\hypersetup{linktocpage=true,colorlinks=true,linkcolor=blue,citecolor=blue,pdfstartview={XYZ 1000 1000 1}}

\usepackage[verbose]{wrapfig}

\usepackage{subfigure}
\renewcommand{\thesubfigure}{\thefigure.\arabic{subfigure}}
\makeatletter
\renewcommand{\p@subfigure}{}
\renewcommand{\@thesubfigure}{\thesubfigure:\hskip\subfiglabelskip}
\makeatother

\newtheorem{theorem}{Theorem}
\newtheorem{lemma}{Lemma}

\newtheorem{proposition}{Proposition}

\theoremstyle{definition}

\newtheorem{example}{Example}

\newcommand{\norm}[1]{\left\|#1\right\|}
\newcommand{\cl}{\mbox{cl}} 
\newcommand{\cv}{\mbox{conv}} 
\newcommand{\near}{\delta} 
\newcommand{\dnear}{\delta_{\Phi}} 
\newcommand{\dfar}{{\not\delta}_{\Phi}} 
\newcommand{\notfar}{\mathop{\not{\delta}}\limits^{\doublewedge}} 
\newcommand{\dcap}{\mathop{\cap}\limits_{\Phi}} 
\newcommand{\sn}{\mathop{\delta}\limits^{\doublewedge}} 
\newcommand{\snd}{\mathop{\delta_{_{\Phi}}}\limits^{\doublewedge}} 
\newcommand{\Cn}{\mbox{\large$\mathfrak{C}$}} 
\newcommand{\Cmn}{\mbox{max}\mathfrak{C}} 
\newcommand{\Cdn}{\mathfrak{C}_{\Phi}}
\newcommand{\Cmdn}{\mbox{max}\mathfrak{C}_{\Phi}}
\newcommand{\Int}{\mbox{int}}
\newcommand{\sdfar}{\stackrel{\not{\text{\normalsize$\delta$}}}{\text{\tiny$\doublevee$}}_{_{\Phi}}} 

\newcommand{\sfar}{\stackrel{\not{\text{\normalsize$\delta$}}}{\text{\tiny$\doublevee$}}} 

\begin{document}

\title{Strongly Proximal Edelsbrunner-Harer Nerves in Vorono\"{i} Tessellations}

\author[J.F. Peters]{J.F. Peters$^{\alpha}$}
\email{James.Peters3@umanitoba.ca, einan@adiyaman.edu.tr}
\address{\llap{$^{\alpha}$\,}Computational Intelligence Laboratory,
University of Manitoba, WPG, MB, R3T 5V6, Canada and
Department of Mathematics, Faculty of Arts and Sciences, Ad\.{i}yaman University, 02040 Ad\.{i}yaman, Turkey}
\author[E. Inan]{E. \.{I}nan$^{\beta}$}
\address{\llap{$^{\beta}$\,} Department of Mathematics, Faculty of Arts and Sciences, Ad\i yaman University, 02040 Ad\i yaman, Turkey and Computational Intelligence Laboratory,
University of Manitoba, WPG, MB, R3T 5V6, Canada}
\thanks{The research has been supported by the Scientific and Technological Research
Council of Turkey (T\"{U}B\.{I}TAK) Scientific Human Resources Development
(BIDEB) under grant no: 2221-1059B211402463 and the Natural Sciences \&
Engineering Research Council of Canada (NSERC) discovery grant 185986.}


\date{}

\dedicatory{Dedicated to the Memory of Som Naimpally}

\begin{abstract}
This paper introduces Edelsbrunner-Harer nerve in collections of Vorono\"{i} regions (called nucleus clusters) endowed with one or more proximity relations.  The main results in this paper are that a maximal nucleus cluster (MNC) in a Vorono\"{i} Tessellation is a strongly proximal Edelsbrunner-Harer nerve, each MNC nerve and the union of the sets in the MNC have the same homotopy type.
\end{abstract}


\maketitle

\section{Introduction}

This paper introduces a variation of Edelsbrunner-Harer nerves which are collections of Vorono\"{i} regions (called nucleus clusters) endowed with one or more proximity relations. Harer-Edelsbrunner nerves are introduced in~\cite[\S III.2, p. 59]{Edelsbrunner2010compTop}.

\setlength{\intextsep}{0pt}
\begin{wrapfigure}[8]{R}{0.35\textwidth}
\begin{minipage}{5.2 cm}
\centering
\includegraphics[width=20mm]{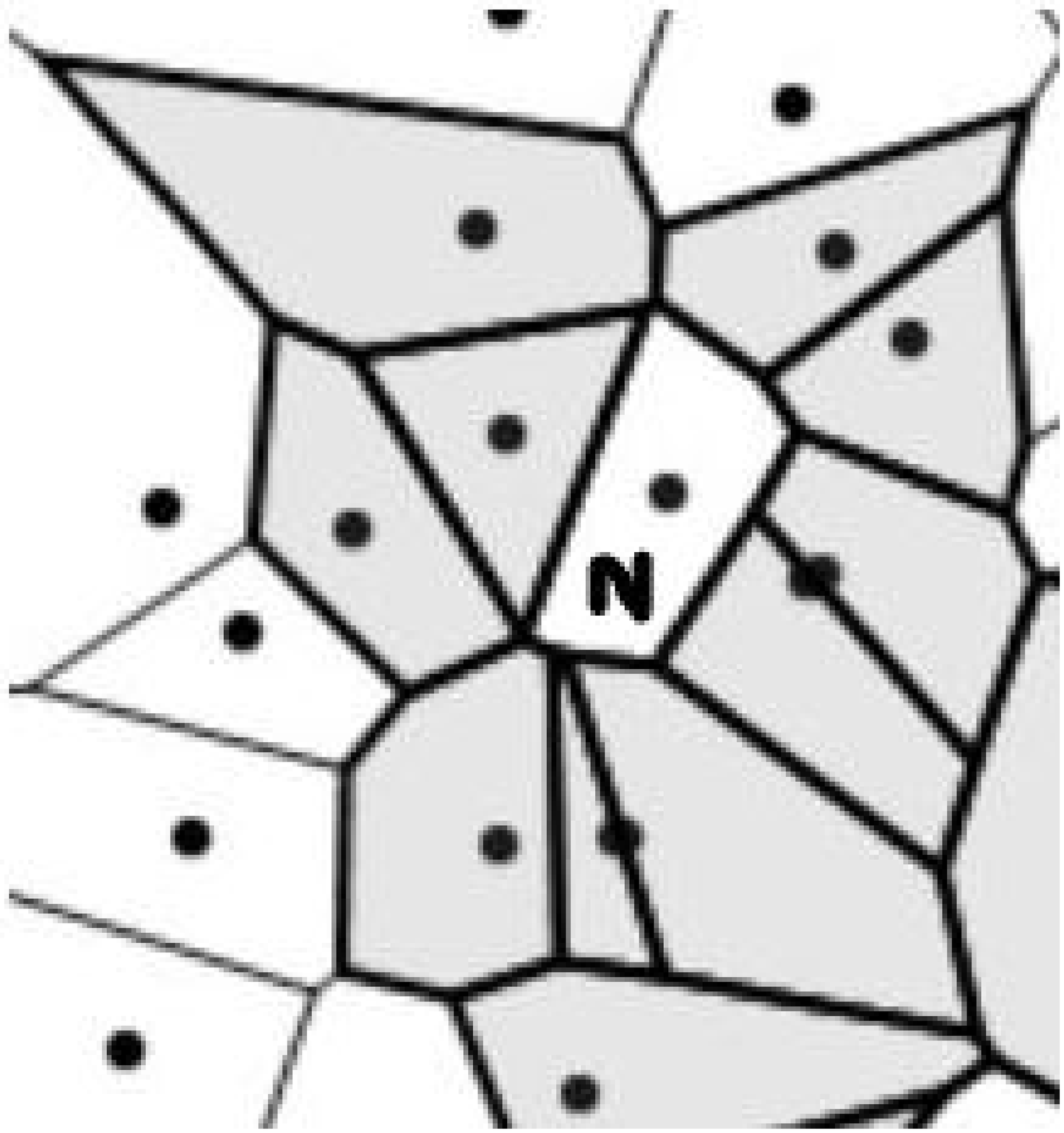}
\caption[]{$\mbox{}$\\ Nucleus Cluster}
\label{fig:nucleusCluster}
\end{minipage}
\end{wrapfigure}

Vorono\"{i} tessellation has great utility and has many applications such as the creation of synthetic poly-crystals, computer graphics~\cite{Fukushige2007VoronoiTessellationPolygonVisibility}, geodesy~\cite{Gold2009VoronoiTessellationGeodesy}, non-parametric sampling~\cite{Villagran2016VoronoiTessellationSampling} and geometric modelling in physics, astrophysics, chemistry and biology~\cite{Qiang2010VoronoiTessellationApplications}.  The form of clustering introduced in this article has proved to be important in the analysis of brain tissue~\cite{Peters2016brainTissueTessellation}, cortical activity and brain symmetries~\cite{Tozzi2016JNeuroSciSymmetries,Duyckaerts2000neuronTessellation} and capillary loss in skeletal and cardiac muscle~\cite{AlShammari2012VoronoiTessellationCapillaryLoss}.  Vorono\"{i} nucleus clustering also has great utility in the study of digital images (see,{\em e.g.},~\cite[\S 1.13]{Peters2016ComputationalProximity},~\cite{Aiyeh2015TAMCSquality},~\cite{Wang2011VoronoiTessellationSegmentation}).  The focus of this paper is not on the applications of MNCs, recently proved to be of great utility~\cite{Tozzi2016JNeuroSciSymmetries,Peters2016brainTissueTessellation}.  Instead, the focus is on maximal nucleus clusters (MNCs) in proximity spaces and MNCs that are strongly proximal Edelsbrunner-Harer nerves.  A proximity space setting for MNCs makes it possible to investigate the strong closeness of subsets in MNCs as well as the spatial and descriptive closeness of MNCs themselves.

\section{Preliminaries}
This section introduces the axioms for traditional as well as strong proximity spaces.   Strong proximities were introduced in~\cite{Peters2015AMSJmanifolds}, elaborated in~\cite{Peters2016ComputationalProximity} (see, also,~\cite{Inan2015}) and are a direct result of earlier work on proximities~\cite{DiConcilio2006,DiConcilio2013mcs,Naimpally1970,Naimpally2009,Naimpally2013}.  

\subsection{Spatial and Descriptive Lodato Proximity}
This section briefly introduces spatial and descriptive forms of proximity that provide a basis for two corresponding forms of strong Lodato proximity introduced in~\cite{Peters2015AMSJmanifolds} and axiomatized in~\cite{Peters2016ComputationalProximity}.  

Let $X$ be a nonempty set. A \emph{Lodato proximity}~\cite{Lodato1962,Lodato1964,Lodato1966} $\delta$ is a relation on the family of sets $2^X$, which satisfies the following axioms for all subsets $A, B, C $ of $X$:\\

\begin{description}
\item[{\rm\bf (P0)}] $\emptyset \not\delta A, \forall A \subset X $.
\item[{\rm\bf (P1)}] $A\ \delta\ B \Leftrightarrow B \delta A$.
\item[{\rm\bf (P2)}] $A\ \cap\ B \neq \emptyset \Rightarrow A \near B$.
\item[{\rm\bf (P3)}] $A\ \delta\ (B \cup C) \Leftrightarrow A\ \delta\ B $ or $A\ \delta\ C$.
\item[{\rm\bf (P4)}] $A\ \delta\ B$ and $\{b\}\ \delta\ C$ for each $b \in B \ \Rightarrow A\ \delta\ C$. \qquad \textcolor{blue}{$\blacksquare$}
\end{description}
\vspace{3mm}
Further $\delta$ is \textit{separated }, if 
\vspace{3mm}
\begin{description}
\item[{\rm\bf (P5)}] $\{x\}\ \delta\ \{y\} \Rightarrow x = y$. \qquad \textcolor{blue}{$\blacksquare$}
\end{description}

\noindent We can associate a topology with the space $(X, \delta)$ by considering as closed sets those sets that coincide with their own closure.  For a nonempty set $A\subset X$, the closure of $A$ (denoted by $\mbox{cl} A$) is defined by,
\[
\mbox{cl} A = \{ x \in X: x\ \delta\ A\}.
\]

The descriptive proximity $\delta_{\Phi}$ was introduced in~\cite{Peters2012ams}.   Let $A,B \subset X$ and let $\Phi(x)$ be a feature vector for $x\in X$, a nonempty set of non-abstract points such as picture points.  $A\ \delta_{\Phi}\ B$ reads $A$ is descriptively near $B$, provided $\Phi(x) = \Phi(y)$ for at least one pair of points, $x\in A, y\in B$.  From this, we obtain the description of a set and the descriptive intersection~\cite[\S 4.3, p. 84]{Naimpally2013} of $A$ and $B$ (denoted by $A\ \dcap\ B$) defined by
\begin{description}
\item[{\rm\bf ($\boldsymbol{\Phi}$)}] $\Phi(A) = \left\{\Phi(x)\in\mathbb{R}^n: x\in A\right\}$, set of feature vectors.
\item[{\rm\bf ($\boldsymbol{\dcap}$)}]  $A\ \dcap\ B = \left\{x\in A\cup B: \Phi(x)\in \Phi(A)\ \mbox{and}\ \Phi(x)\in \Phi(B)\right\}$.
\qquad \textcolor{blue}{$\blacksquare$}
\end{description}
Then swapping out $\near$ with $\dnear$ in each of the Lodato axioms defines a descriptive Lodato proximity. 

That is, a \textit{descriptive Lodato proximity $\dnear$} is a relation on the family of sets $2^X$, which satisfies the following axioms for all subsets $A, B, C $ of $X$.\\

\begin{description}
\item[{\rm\bf (dP0)}] $\emptyset\ \dfar\ A, \forall A \subset X $.
\item[{\rm\bf (dP1)}] $A\ \dnear\ B \Leftrightarrow B\ \dnear\ A$.
\item[{\rm\bf (dP2)}] $A\ \dcap\ B \neq \emptyset \Rightarrow\ A\ \dnear\ B$.
\item[{\rm\bf (dP3)}] $A\ \dnear\ (B \cup C) \Leftrightarrow A\ \dnear\ B $ or $A\ \dnear\ C$.
\item[{\rm\bf (dP4)}] $A\ \dnear\ B$ and $\{b\}\ \dnear\ C$ for each $b \in B \ \Rightarrow A\ \dnear\ C$. \qquad \textcolor{blue}{$\blacksquare$}
\end{description}
\vspace{3mm}
Further $\dnear$ is \textit{descriptively separated }, if 
\vspace{3mm}
\begin{description}
\item[{\rm\bf (dP5)}] $\{x\}\ \dnear\ \{y\} \Rightarrow \Phi(x) = \Phi(y)$ ($x$ and $y$ have matching descriptions). \qquad \textcolor{blue}{$\blacksquare$}
\end{description}
\vspace{3mm}

\noindent The pair $\left( X,\dnear \right)$ is called a \emph{descriptive proximity space}.   Unlike the Lodato Axiom (P2), the converse of the descriptive Lodato Axiom (dP2) also holds.

\begin{proposition}\label{prop:dnear}
Let $\left(X,\dnear\right)$ be a descriptive proximity space, $A,B\subset X$.  Then $A\ \dnear\ B \Rightarrow A\ \dcap\ B\neq \emptyset$.
\end{proposition}
\begin{proof}
$A\ \dnear\ B \Leftrightarrow$ there is at least one $x\in A, y\in B$ such that $\Phi(x)=\Phi(y)$ (by definition of $A\ \dnear\ B$)\  Hence, $A\ \dcap\ B\neq \emptyset$.
\end{proof}


\subsection{Spatial and Descriptive Strong Proximities}
This section briefly introduces spatial strong proximity between nonempty sets and descriptive strong Lodato proximity.

Nonempty sets $A,B$ in a topological space $X$ equipped with the relation $\sn$, are \emph{strongly near} [\emph{strongly contacted}] (denoted $A\ \sn\ B$), provided the sets have at least one point in common.   The strong contact relation $\sn$ was introduced in~\cite{Peters2015JangjeonMSstrongProximity} and axiomatized in~\cite{PetersGuadagni2015stronglyNear},~\cite[\S 6 Appendix]{Guadagni2015thesis}.

Let $X$ be a topological space, $A, B, C \subset X$ and $x \in X$.  The relation $\sn$ on the family of subsets $2^X$ is a \emph{strong proximity}, provided it satisfies the following axioms.

\begin{description}
\item[{\rm\bf (snN0)}] $\emptyset\ \sfar\ A, \forall A \subset X $, and \ $X\ \sn\ A, \forall A \subset X$.
\item[{\rm\bf (snN1)}] $A \sn B \Leftrightarrow B \sn A$.
\item[{\rm\bf (snN2)}] $A\ \sn\ B$ implies $A\ \cap\ B\neq \emptyset$. 
\item[{\rm\bf (snN3)}] If $\{B_i\}_{i \in I}$ is an arbitrary family of subsets of $X$ and  $A \sn B_{i^*}$ for some $i^* \in I \ $ such that $\Int(B_{i^*})\neq \emptyset$, then $  \ A \sn (\bigcup_{i \in I} B_i)$ 
\item[{\rm\bf (snN4)}]  $\mbox{int}A\ \cap\ \mbox{int} B \neq \emptyset \Rightarrow A\ \sn\ B$.  
\qquad \textcolor{blue}{$\blacksquare$}
\end{description}

\noindent When we write $A\ \sn\ B$, we read $A$ is \emph{strongly near} $B$ ($A$ \emph{strongly contacts} $B$).  The notation $A\ \notfar\ B$ reads $A$ is not strongly near $B$ ($A$ does not \emph{strongly contact} $B$). For each \emph{strong proximity} (\emph{strong contact}), we assume the following relations:
\begin{description}
\item[{\rm\bf (snN5)}] $x \in \Int (A) \Rightarrow x\ \sn\ A$ 
\item[{\rm\bf (snN6)}] $\{x\}\ \sn \{y\}\ \Leftrightarrow x=y$  \qquad \textcolor{blue}{$\blacksquare$} 
\end{description}

For strong proximity of the nonempty intersection of interiors, we have that $A \sn B \Leftrightarrow \Int A \cap \Int B \neq \emptyset$ or either $A$ or $B$ is equal to $X$, provided $A$ and $B$ are not singletons; if $A = \{x\}$, then $x \in \Int(B)$, and if $B$ too is a singleton, then $x=y$. It turns out that if $A \subset X$ is an open set, then each point that belongs to $A$ is strongly near $A$.  The bottom line is that strongly near sets always share points, which is another way of saying that sets with strong contact have nonempty intersection.   Let $\near$ denote a traditional proximity relation~\cite{Naimpally1970}. \\

\begin{figure}[!ht]
\centering
\includegraphics[width=55mm]{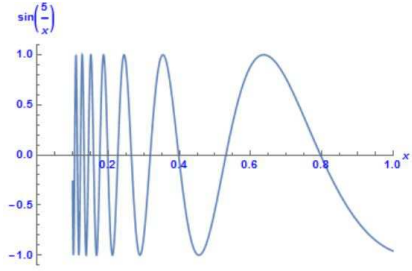}
\caption[]{Near Sets: $A = \left\{(x,0): 0.1\leq x\leq 1\right\},B = \left\{(x,sin(5/x)):0.1\leq x\leq 1\right\}$}
\label{fig:nearSetsStrong}
\end{figure}
$\mbox{}$\\
\vspace{2mm}

Next, consider a proximal form of a Sz\'{a}z relator~\cite{Szaz1987}.  A \emph{proximal relator} $\mathscr{R}$ is a set of relations on a nonempty set $X$~\cite{Peters2016relator}.  The pair $\left(X,\mathscr{R}\right)$ is a proximal relator space.  The connection between $\sn$ and $\near$ is summarized in Prop.~\ref{thm:sn-implies-near}.

\begin{proposition}\label{thm:sn-implies-near}
Let $\left(X,\left\{\near,\dnear,\sn\right\}\right)$ be a proximal relator space, $A,B\subset X$.  Then 
\begin{compactenum}[1$^o$]
\item $A\ \sn\ B \Rightarrow A\ \near\ B$.
\item $A\ \sn\ B \Rightarrow A\ \dnear\ B$.
\end{compactenum}
\end{proposition}
\begin{proof}$\mbox{}$\\
1$^o$: From Axiom (snN2), $A\ \sn\ B$ implies $A\ \cap\ B\neq \emptyset$, which implies $A\ \near\ B$ (from Lodato Axiom (P2)).\\
2$^o$: From 1$^o$, there are $x\in A, y\in B$\ common to $A$ and $B$.  Hence, $\Phi(x) = \Phi(y)$, which implies $A\ \dcap\ B\neq \emptyset$.  Then, from the descriptive Lodato Axiom (dP2), $A\ \dcap\ B \neq \emptyset \Rightarrow\ A\ \dnear\ B$. This gives the desired result.
\end{proof}

\begin{example}
Let $X$ be a topological space endowed with the strong proximity $\sn$ and $A = \left\{(x,0): 0.1\leq x\leq 1\right\}$,$B = \left\{(x,sin(5/x)):0.1\leq x\leq 1\right\}$.  In this case, $A,B$ represented by
Fig.~\ref{fig:nearSetsStrong} are strongly near sets with many points in common.  
\qquad \textcolor{blue}{$\blacksquare$}
\end{example}

The descriptive strong proximity $\snd$ is the descriptive counterpart of $\sn$.
To obtain a \emph{descriptive strong Lodato proximity} (denoted by \emph{\bf dsn}), we swap out $\dnear$ in each of the descriptive Lodato axioms with the descriptive strong proximity $\snd$.  

Let $X$ be a topological space, $A, B, C \subset X$ and $x \in X$.  The relation $\snd$ on the family of subsets $2^X$ is a \emph{descriptive strong Lodato proximity}, provided it satisfies the following axioms.
\vspace{2mm}
\begin{description}
\item[{\rm\bf (dsnP0)}] $\emptyset\ {\sdfar}\ A, \forall A \subset X $, and \ $X\ \snd\ A, \forall A \subset X$.
\item[{\rm\bf (dsnP1)}] $A\ \snd\ B \Leftrightarrow B\ \snd\ A$.
\item[{\rm\bf (dsnP2)}] $A\ \snd\ B$ implies $A\ \dcap\ B\neq \emptyset$.  
\item[{\rm\bf (dsnP4)}] $\mbox{int}A\ \dcap\ \mbox{int} B \neq \emptyset \Rightarrow A\ \snd\ B$.  
\qquad \textcolor{blue}{$\blacksquare$}
\end{description}

\noindent When we write $A\ \snd\ B$, we read $A$ is \emph{descriptively strongly near} $B$.
For each \emph{descriptive strong proximity}, we assume the following relations:
\begin{description}
\item[{\rm\bf (dsnP5)}] $\Phi(x) \in \Phi(\Int (A)) \Rightarrow x\ \snd\ A$. 
\item[{\rm\bf (dsnP6)}] $\{x\}\ \snd\ \{y\} \Leftrightarrow \Phi(x) = \Phi(y)$.  
\qquad \textcolor{blue}{$\blacksquare$} 
\end{description}

So, for example, if we take the strong proximity related to non-empty intersection of interiors, we have that $A\ \snd\ B \Leftrightarrow \Int A\ \dcap\ \Int B \neq \emptyset$ or either $A$ or $B$ is equal to $X$, provided $A$ and $B$ are not singletons; if $A = \{x\}$, then $\Phi(x) \in \Phi(\Int(B))$, and if $B$ is also a singleton, then $\Phi(x)=\Phi(y)$. 

The connections between $\snd,\dnear$ are summarized in Prop.~\ref{thm:sn-implies-dnear}.  
 
\begin{proposition}\label{thm:sn-implies-dnear}
Let $\left(X,\left\{\sn,\dnear,\snd\right\}\right)$ be a proximal relator space, $A,B\subset X$.  Then 
\begin{compactenum}[1$^o$]
\item For $A,B$ not equal to singletons, $A\ \snd\ B \Rightarrow \Int A\ \dcap\ \Int B\neq \emptyset \Rightarrow \Int A\ \dnear\ \Int B$.
\item $A\ \sn \ B \Rightarrow (\Int A\ \dcap\ \Int B)\ \snd \ B$.
\item $A\ \snd\ B \Rightarrow A\ \dnear\ B$.
\end{compactenum}
\end{proposition}
\begin{proof}$\mbox{}$\\
1$^o$:
$A\ \snd\ B$ implies that interior of $A$ is descriptively near the interior of $B$.  Consequently, $\Int A\ \dcap\ \Int B\neq \emptyset$.  Hence, from Axiom (dP2), $\Int A\ \dnear\ \Int B$.\\
1$^o\Rightarrow\ 2^o$.
3$^o$: Immediate from Axioms (dsnP2) and (dP2).
\end{proof}

\subsection{Vorono\"{i} regions}
Let $E$ be the Euclidean plane, $S\subset E$ (set of mesh generating points), $s\in S$.  A Vorono\"{i} region (denoted by $V(s)$) is defined by
\[
V(s) = \left\{x\in E: \norm{x - s}\leq \norm{x - q}, \mbox{for all}\ q\in S\right\}\ \mbox{(Vorono\"{i} region)}.
\]

Let $X$ be a collection of Vorono\"{i} regions containing $N$, endowed with the strong proximity $\sn$.  A nucleus mesh cluster (denoted by $\Cn\ N$) in a Vorono\"{i} tessellation is defined by
\[
\Cn\ N = \left\{A\in X: \cl\ A\ \sn\ N\right\}\ \mbox{(Vorono\"{i} mesh nucleus cluster)}.\\
\]

\begin{example}
A partial view of a Vorono\"{i} tessellation of a plane surface is shown in Fig.~\ref{fig:nucleusCluster}.  The Vorono\"{i} region $N$ in this tessellation is the nucleus of a mesh cluster containing all of those polygons adjacent to $N$.  
\qquad \textcolor{blue}{\Squaresteel}
\end{example}

A \emph{concrete} (\emph{physical}) set $A$ of points $p$ that are described by their location and physical characteristics, {\em e.g.}, gradient orientation (angle of the tangent to $p$.  Let $\varphi(p)$ be the gradient orientation of $p$.   For example, each point $p$ with coordinates $(x,y)$ in the concrete subset $A$ in the Euclidean plane is described by a feature vector of the form $(x,y,\varphi(p(x,y))$. Nonempty concrete sets $A$ and $B$ have descriptive strong proximity (denoted $A\ \snd\ B$), provided $A$ and $B$ have points with matching descriptions.
In a region-based, descriptive proximity extends to both abstract and concrete sets~\cite[\S 1.2]{Peters2016ComputationalProximity}.  For example, every subset $A$ in the Euclidean plane has features such as area and diameter.  Let $(x,y)$ be the coordinates of the centroid $m$ of $A$.  Then $A$ is described by feature vector of the form $(x,y,area, diameter)$.  Then regions $A,B$ have descriptive proximity (denoted $A\ \snd\ B$), provided $A$ and $B$ have matching descriptions.

The notion of strongly proximal regions extends to convex sets.  A nonempty set $A$ is a \emph{convex set} (denoted $\cv A$), provided, for any pair of points $x,y\in A$, the line segment $\overline{xy}$ is also in $A$.  The empty set $\emptyset$ and a one-element set $\left\{x\right\}$ are convex by definition.  Let $\mathscr{F}$ be a family of convex sets.  From the fact that the intersection of any two convex sets is convex~\cite[\S 2.1, Lemma A]{Edelsbrunner2014}, it follows that
\[
\mathop{\bigcap}\limits_{A\in\mathscr{F}} A\ \mbox{is a convex set}.
\]
Convex sets $\cv A, \cv B$ are strongly proximal (denote $\cv A \sn\ \cv B$), provided $\cv A, \cv B$ have points in common.  Convex sets $\cv A, \cv B$ are descriptively strongly proximal (denoted $\cv A \snd\ \cv B$), provided $\cv A, \cv B$ have matching descriptions. 

Let $X$ be a Vorono\"{i} tessellation of a plane surface equipped with the strong proximity $\sn$ and descriptive strong proximity $\snd$ and let $A,N\in X$ be Vorono\"{i} regions.   The pair $\left(X,\left\{\sn,\snd\right\}\right)$ is an example of a proximal relator space~\cite{Peters2016relator}.
The two forms of nucleus clusters (ordinary nucleus cluster denoted by $\Cn$) and descriptive nucleus clusters are examples of mesh nerves~\cite[\S 1.10, pp. 29ff]{Peters2016ComputationalProximity}, defined by
\begin{align*}
\Cn N &= \left\{A\in X: A\ \sn\ N\right\}\ \mbox{(nucleus cluster)}.\\
\Cdn N  &= \left\{A\in X: A\ \snd\ N\right\}\ \mbox{(descriptive nucleus cluster)}.
\end{align*}

A nucleus cluster is \emph{maximal} (denoted by $\Cmn N$), provided $N$ has the highest number of adjacent polygons in a tessellated surface (more than one maximal cluster in the same mesh is possible).  Similarly, a descriptive nucleus cluster 
is maximal (denoted by $\Cmdn N$), provided $N$ has the highest number of polygons in a tessellated surface descriptively near $N$, {\em i.e.}, the description of each $A\in \Cmdn N$ matches the description of nucleus $N$ and the number of polygons descriptively near $N$ is maximal (again, more than one $\Cmdn N$ is possible in a Vorono\"{i} tessellation).\\
\vspace{3mm}

\begin{figure}[!ht]
\centering
\includegraphics[width=65mm]{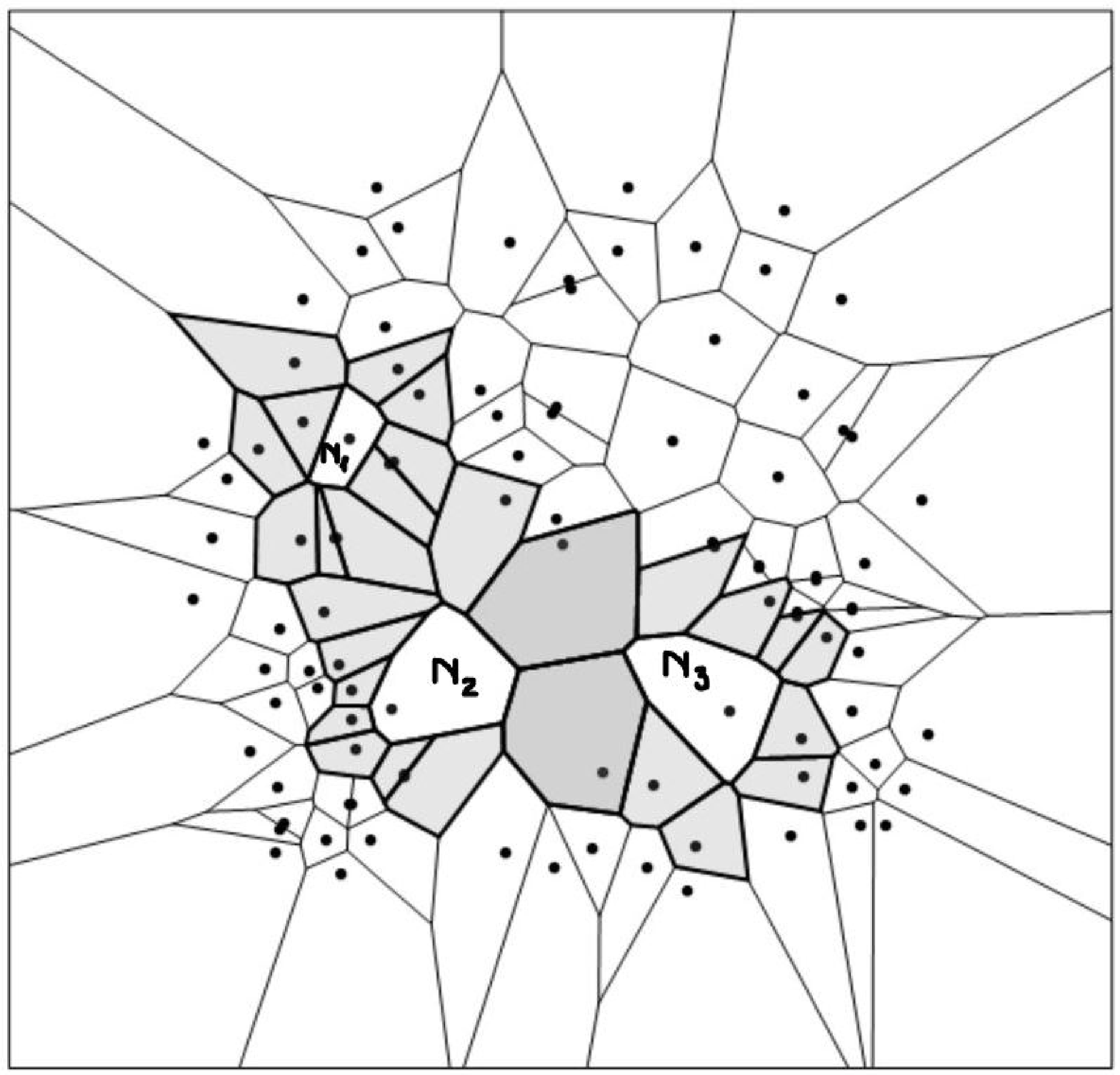}
\caption[]{ $\Cn\ N_1\ \sn\ \Cn\ N_2\ \mbox{and}\ \Cn\ N_1\ \snd\ \Cn\ N_2$}
\label{fig:N1N2N3maxClusters}
\end{figure}

\begin{example}$\mbox{}$\\
Let $X$ be the collection of Vorono\"{i} regions in a tessellation of a subset of the Euclidean plane shown in Fig.~\ref{fig:N1N2N3maxClusters} with nuclei $N_1,N_2,N_3\in X$. In addition, let $2^X$ be the family of all subsets of Vorono\"{i} regions in $X$ containing maximal nucleus clusters $\Cn N_1,\Cn N_2,\Cn N_3\in 2^X$ in the tessellation.  Then, for example, $\Int\Cn N_2\ \cap\ \Int\Cn N_3\neq\emptyset$, since $\Cn N_2,\Cn N_3$ share Vorono\"{i} regions.  Hence, $\Cn N_2\ \sn\ \Cn N_3\neq\emptyset$ (from Axiom (snN4)). Similarly, 
$\Cn N_1\ \sn\ \Cn N_2$.

Let $\Phi(A)$ the description of a Vorono\"{i} equal the number of sides of $A\in X$.  Since the nuclei $N_1,N_2,N_3$ have matching descriptions, $\Int\Cn N_1\ \dcap\ \Int\Cn N_2\neq \emptyset$. Consequently, $\Cn N_1\ \snd\ \Cn N_2$ (from Axiom (dsnP4)).  Similarly, $\Cn N_1\ \snd\ \Cn N_3$ and $\Cn N_2\ \snd\ \Cn N_3$.
\qquad \textcolor{blue}{\Squaresteel}
\end{example}

\setlength{\intextsep}{0pt}
\begin{wrapfigure}[13]{R}{0.30\textwidth}
\begin{minipage}{3.8 cm}
\centering
\begin{pspicture}
 (-0.5,-1.0)(3.5,2.5) 
\providecommand{\PstPolygonNode}{%
  \psdots[showpoints=true,linecolor=black](1;\INode)}
\PstPentagon[unit=0.75]
\psline[showpoints=true,linecolor=gray](-0.55,-0.83)(-1.10,-0.83) 
\psline[showpoints=true,linecolor=gray](-0.55,-0.83)(-0.30,0.12)
\psline[showpoints=true,linecolor=gray](-1.10,-0.83)(-1.18,0.12)
\psline[showpoints=true,linecolor=gray](-0.75,1.50)(0.08,2.12) 
\psline[showpoints=true,linecolor=gray](0.08,2.12)(0.38,1.82)
\psline[showpoints=true,linecolor=gray](0.38,1.82)(-0.00,1.00)
\rput(-1.05,0.80){\Large \textcolor{black}{$\boldsymbol{N}$}}
\rput(-0.20,1.50){\Large \textcolor{black}{$\boldsymbol{A_{_{1}}}$}}
\rput(-0.80,-0.30){\Large \textcolor{black}{$\boldsymbol{A_{_{2}}}$}}
\end{pspicture}
\caption[]{$\mbox{}$\\ {MNC Spokes}}
\label{fig:MNCnerve}
\end{minipage}
\end{wrapfigure}
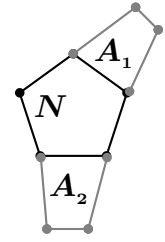

\section{Main Results}
Homotopy types are introduced in~\cite[\S III.2]{Edelsbrunner1999} and lead to significant results for Vorono\"{i} maximal nucleus clusters.  

Let $f,g:X\longrightarrow Y$ be two continuous maps.  A \emph{homotopy} between $f$ and $g$ is a continuous map $H:X\times[0,1]\longrightarrow Y$ so that $H(x,0) = f(x)$ and $H(x,1) = g(x)$.  The sets $X$ and $Y$ are \emph{homotopy equivalent}, provided there are continuous maps $f: X\longrightarrow Y$ and $g:Y\longrightarrow X$ such that $g\circ f \simeq \mbox{id}_X$ and $f\circ g \simeq \mbox{id}_Y$.  This yields an equivalence relation $X\simeq Y$.  In addition, $X$ and $Y$ have the same \emph{homotopy type}, provided $X$ and $Y$ are homotopy equivalent.  

Let $\mathscr{F}$ be a finite collection of sets.  An \emph{Edelsbrunner-Harer nerve} (denoted by $\mbox{Nrv}\ \mathscr{F}$) consists of all nonempty subcollections of $\mathscr{F}$ that have a nonvoid common intersection, {\em i.e.},
\[
\mbox{Nrv}\ \mathscr{F} = \left\{X\in \mathscr{F}: \bigcap X\neq \emptyset\right\}.
\]

Let $\mathscr{F}_{_{MNC}}$ be a collection of polygons in a Vorono\"{i} MNC endowed with the strong proximity $\sn$, $A$ be a Vorono\"{i} region in a MNC $\Cn N$ with nucleus $N$ and let subscollection $\mathscr{S}_A = \left\{A,N\right\}\in\Cn N$.  The pair $\left(\mathscr{F}_{_{MNC}},\sn\right)$ is a proximity space.  For each MNC $\Cn N$ endowed with $\sn$, the nucleus $N$ together with its adjacent polygons is a Vorono\"{i} structure (denoted by $\mbox{Nrv}\mathscr{F}_{_{MNC}}$) defined by
\[
\mbox{Nrv}\mathscr{F}_{_{MNC}} = \left\{\mathscr{S}\in \Cn N: \left(N,A\right)\in \sn\ \mbox{for}\ A\in \mathscr{S}\right\}\ \mbox{(MNC nerve)}.
\]

Each pair $\left(N,A\right)\in \sn$ in $\mbox{Nrv}\mathscr{F}_{_{MNC}}$ is called a \emph{\bf spoke} (denoted by $\mathscr{S}_A$), with a shape similar to the spoke in a wheel.  A spoke contains a Vorono\"{i} region $A\in \Cn N$ that shares an edge with $N$.  Hence, the $A\ \sn\ N$, {\em i.e.}, there is a strong proximity between the subsets in a spoke.

\begin{example}\label{ex:MNCnerve}
A pair of spokes $\mathscr{S}_{A_1},\mathscr{S}_{A_2}$ in a fragment of an MNC $\Cn N$ with nucleus $N$ is represented in Fig.~\ref{fig:MNCnerve}.  
\qquad \textcolor{blue}{\Squaresteel}
\end{example}

Every MNC $\Cn N$ is a finite collection of closed convex sets in the Euclidean plane.  Let $\Cn N$ be endowed with the strong proximity $\sn$.  All non-nucleus polygons in $\Cn N$ share an edge with $N$.  The collection of spokes $\mathscr{S}_A\in \Cn N$ each contain the nucleus $N$, which is common to all of the spokes, {\em i.e.}, the spokes in $\mbox{Nrv}\mathscr{F}_{_{MNC}}$ have a nonvoid common intersection.  Let $\mathscr{S}_A,\mathscr{S}_{A'}$ be spokes in $\Cn N$ that share nucleus $N$.  Consequently, $\Int\left(\mathscr{S}_A\right)\ \cap\ \Int\left(\mathscr{S}_{A'}\right) \neq\emptyset$ implies $\mathscr{S}_A\ \sn\ \mathscr{S}_{A'}$ (from Axiom (snN4)).  Hence, $\mbox{Nrv}\mathscr{F}_{_{MNC}}$ is an Edelsbrunner-Harer nerve.  From this, we obtain the result in Lemma~\ref{lem:MNCnerves}.

\begin{lemma}\label{lem:MNCnerves}
Let $\mathscr{F}_{_{MNC}}$ be a collection of polygons in a Vorono\"{i} MNC endowed with the strong proximity $\sn$.
The structure $\mbox{Nrv}\mathscr{F}_{mnc}$ is an Edelsbrunner-Harer nerve.
\end{lemma}
\begin{proof}
Let $\mathscr{S}_A,\mathscr{S}_{A'}$ be a pair of spokes in a maximal nucleus cluster MNC $\Cn N$.  Since $\mathscr{S}_A\ \sn\ \mathscr{S}_{A'}$ have $N$ in common, $\mathscr{S}_A\sn \mathscr{S}_{A'}$ implies $\mathscr{S}_A\cap \mathscr{S}_{A'}\neq \emptyset$ (from Axiom (snN2)).  This holds true for all spokes in $\Cn N$.  Consequently, $\mathop{\bigcap}\limits_{\mathscr{S}_A\in \Cn N} \mathscr{S}_A\neq \emptyset$.  Hence, the structure $\mbox{Nrv}\mathscr{F}_{_{MNC}}$ is an Edelsbrunner-Harer nerve.
\end{proof}

\begin{theorem}\label{thm:sn-nerve}
Let $\left(\mbox{Nrv}\mathscr{F}_{_{MNC}},\left\{\near,\dnear,\sn\right\}\right)$ be a proximal relator space, spokes $\mathscr{S}_A,\mathscr{S}_{A'}\in \mbox{Nrv}\mathscr{F}_{_{MNC}}$.  Then 
\begin{compactenum}[1$^o$]
\item $\mathscr{S}_A\ \sn\ \mathscr{S}_{A'} \Rightarrow \mathscr{S}_A\ \near\ \mathscr{S}_{A'}$.
\item $\mathscr{S}_A\ \sn\ \mathscr{S}_{A'} \Rightarrow \mathscr{S}_A\ \dnear\ \mathscr{S}_{A'}$.
\end{compactenum}
\end{theorem}
\begin{proof}$\mbox{}$\\
1$^o$: From Lemma~\ref{lem:MNCnerves}, $\mbox{Nrv}\mathscr{F}_{_{MNC}}$ is an Edelsbrunner-Harer nerve.  Consequently, $\mathscr{S}_A\ \sn\ \mathscr{S}_{A'}$ for every pair of spokes $\mathscr{S}_A,\mathscr{S}_{A'}$ in the nerve. 
Then, $\mathscr{S}_A\ \sn\ \mathscr{S}_{A'}$ implies $\mathscr{S}_A\ \cap\ \mathscr{S}_{A'}\neq \emptyset$, which implies $\mathscr{S}_A\ \near\ \mathscr{S}_{A'}$ (from Prop.~\ref{thm:sn-implies-near}).\\
2$^o$: Spokes $\mathscr{S}_A,\mathscr{S}_{A'}$ have nucleus $N$ in common.  Hence, $\mathscr{S}_A\ \dcap\ \mathscr{S}_{A'}\neq \emptyset$.  Then, from Prop.~\ref{thm:sn-implies-near}, $\mathscr{S}_A\ \dcap\ \mathscr{S}_{A'} \neq \emptyset \Rightarrow\ \mathscr{S}_A\ \dnear\ \mathscr{S}_{A'}$. This gives the desired result for each pair of spokes in the nerve.
\end{proof}

\begin{theorem}\label{EHnerve}{\rm ~\cite[\S III.2, p. 59]{Edelsbrunner1999}}
Let $\mathscr{F}$ be a finite collection of closed, convex sets in Euclidean space.  Then the nerve of $\mathscr{F}$ and the union of the sets in $\mathscr{F}$ have the same homotopy type.
\end{theorem}

\begin{theorem}
Let the nucleus cluster $\Cn N$ be a finite collection of closed, convex sets in a Vorono\"{i} mesh $V$ in the Euclidean plane.  The nerve $\mbox{Nrv}\mathscr{F}_{MNC}$ in $\Cn N$ and the union of the sets in $\Cn N$ have the same homotopy type.
\end{theorem}
\begin{proof}
Let $\Cn N$ be a MNC be nucleus $N$ in a Vorono\"{i} mesh.   From Lemma~\ref{lem:MNCnerves}, $\mbox{Nrv}\mathscr{F}_{_{MNC}}$ is a Edelsbrunner-Harer nerve.  From Theorem~\ref{EHnerve}, we have that the union of the sets in $\Cn N$ and $\mbox{Nrv}\mathscr{F}_{_{MNC}}$ have the same homotopy type.
\end{proof}
   
\begin{theorem}\label{lem:sndMNCnerves}
Let $X$ be a finite collection of MNC Edelsbrunner-Harer nerves $\mbox{Nrv}\mathscr{F}_{_{MNC}}$ in a Vorono\"{i} mesh with nuclei $N$ in the Euclidean plane and let $X$ be equipped with the relator $\left\{\sn,\snd\right\}$ with strongly close mesh nerves.  Each nucleus $N$ has a description $\Phi(N) = \mbox{number of sides of}\ N$.  Then $\mathop{\bigcap}\limits_{\Phi}\mbox{Nrv}\mathscr{F}_{_{MNC}} \neq \emptyset$.
\end{theorem}
\begin{proof}
Each $\mbox{Nrv}\mathscr{F}_{_{MNC}}$ is a collection of Vorono\"{i} regions containing a nucleus polygon $N$ with the same number of sides, since $\mbox{Nrv}\mathscr{F}_{_{MNC}}\in \Cn N$, which is maximal.
Let $\mathscr{N},\mathscr{N'}\in X$ be nerves with nuclei $N_1,N_2$ in maximal nucleus clusters.  $\Phi(N_1) = \Phi(N_2)$, since $\Cn N_1,\Cn N_2$ are maximal, {\em i.e.}, $N_1,N_2$ have same number of sides.  This means that all nuclei in $\mathscr{N},\mathscr{N'}$ have the same description.  Consequently, $\mathscr{N}\ \snd\ \mathscr{N'}$ implies $\Int N_1\ \dcap\ \Int N_2\neq \emptyset$ (from Axiom (dsnP2)).  Hence,
$N_1\ \snd\ N_2$ implies $\mathscr{N}\ \dnear\ \mathscr{N'}$ (from Prop.~\ref{thm:sn-implies-dnear}).  Then $\mathscr{N}\ \dnear\ \mathscr{N'}$ implies $\mathscr{N}\ \dcap\ \mathscr{N'}\neq\emptyset$ (from Prop.~\ref{prop:dnear}).  Therefore, $\mathop{\bigcap}\limits_{\Phi}\mbox{Nrv}\mathscr{F}_{_{MNC}} \neq \emptyset$.
\end{proof}

\bibliographystyle{amsplain}
\bibliography{NSrefs}

\end{document}